\documentclass{amsart}
\usepackage[english]{babel}
\usepackage{amsfonts, amsmath, amsthm, amssymb,amscd,indentfirst}

\usepackage{color}
\definecolor{dark_purple}{rgb}{0.4, 0.0, 0.4}

\usepackage{esint}
\usepackage{amsmath,amssymb,latexsym,indentfirst}
\usepackage{times}
\usepackage{palatino}
\usepackage[numbers]{natbib}

\newtheorem{theorem}{Theorem}[section]

\newtheorem{proposition}{Proposition}[section]
\newtheorem{lemma}{Lemma}[section]
\newtheorem{definition}{Definition}[section]

\newtheorem{remark}{Remark}[section]

\def\Cc{{\cal C}}

\def\Cc{{\cal C}}
\def\Bc{{\mathcal B}}

\def\Hc{{\mathcal H}}

\def\Dc{{\mathcal D}}

\def\Kc{{\mathcal K}}

\def\Wc{{\mathcal W}}

\def\Cc{{\cal C}}

\def\Cc{\mathcal{C}}

\begin{document}

\title[On the Brown-York quasi-local mass]{On the limiting behavior  of the Brown-York quasi-local mass in asymptotically hyperbolic manifolds}

\author{Ezequiel Barbosa}
\address{Universidade Federal de Minas Gerais, Departmento de Matem\'atica, Caixa Postal 702, 30123-970, Belo Horizonte, MG, Brazil}
\email{ezequiel@mat.ufmg.br}
\author{Levi Lopes de Lima}
\address{Universidade Federal d Cear\'a,
Departamento de Matem\'atica, Campus do Pici, Av. Humberto Monte, s/n, Bloco 914, 60455-760,
Fortaleza/CE, Brazil.}
\email{levi@mat.ufc.br}
\email{fred@mat.ufc.br}
\author{Frederico Gir\~ao}

\thanks{The first  and second authors were partially supported by CNPq/Brazil.}

\begin{abstract}
We show that the limit at infinity of the vector-valued Brown-York-type quasi-local mass along any coordinate exhaustion of an asymptotically hyperbolic $3$-manifold satisfying the relevant energy condition on the scalar curvature has the conjectured causal character. Our proof uses spinors and relies on  a Witten-type formula expressing the asymptotic limit of this quasi-local mass 
as a bulk integral which manifestly has the right sign under the above assumptions. 
In the spirit of recent work by Hijazi, Montiel and Raulot, we also provide another proof of this result which uses the theory of boundary value problems for Dirac operators on compact domains to 
show that a certain quasi-local mass, which converges to the Brown-York mass in the asymptotic limit,  has the expected causal character under suitable geometric assumptions.
\end{abstract}

\maketitle

\section{Introduction}\label{int}

Let $\overline M$ be a compact $3$-manifold with boundary $S$ and let $M$ be
the interior of $\overline M$. A nonnegative smooth function $\rho:\overline M\to \mathbb R$ such that $\rho^{-1}(0)=S$ and $d\rho|_S\neq 0$ is called a {\em defining function}. Let $g$ be a Riemannian metric on $M$. We say that $(M,g)$ is {\em conformally compact} if for any defining function $\rho$ the metric $\overline g=\rho^2g$ extends to a smooth metric on $\overline M$. The restriction $\overline g|_S$ defines a metric which changes by a conformal factor if the defining function is changed.
Thus, the conformal class of $\overline g|_S$, called the {\em conformal infinity} of $(M,g)$, is well defined

We say that $(M,g)$ as above is {\em weakly asymptotically hyperbolic} if $|d\rho|_{\overline g}=1$ along $S$.
This means that the sectional curvatures of $(M,g)$ converge to $-1$ as one approaches $S$. In this case, if $h_0$ is a metric on $S$ in  the given conformal class, there is a unique defining function $\rho$ in a collar neighborhood of $S$ so that
\begin{equation}\label{expan}
g=\sinh^{-2}\rho\left(d\rho^2+h_{\rho}\right),
\end{equation}
where $h_{\rho}$ is a $\rho$-dependent family of metrics with $h_{\rho}|_{\rho=0}=h_0$; see  \cite{AD} and \cite{MP} for further details.
We will denote by $\nabla$ the Levi-Civita connection of $g$.

\begin{definition}\label{defasym2}
Under the conditions above, we say that $(M,g)$ is {\em asymptotically hyperbolic} if $h_0$ is a round metric on the sphere $\mathbb S^2$ and if, in a collar neighborhood of the conformal infinity, the following asymptotic expansion holds as $\rho\to 0$:
\begin{equation}\label{asymexpcoll}
h_{\rho}=h_0+\frac{\rho^3}{3}h+e,
\end{equation}
where $h$ is a symmetric $2$-tensor on $\mathbb S^{2}$ and the remainder $e$
satisfies 
\begin{equation}
\label{remainder}
|e|+|\nabla e|+|\nabla^2e |+|\nabla^3 e|=O(\rho^4).
\end{equation}
\end{definition}

The tensor $h$ above measures the  deviation of $g$ to the background hyperbolic metric $g_0=\sinh^{-2}(d\rho^2+h_0)$ as $\rho\to 0$. This led Wang \cite{Wa} to define the {\em mass} of $(M,g)$
as
\begin{equation}\label{massvec}
\Upsilon_{(M,g)}=\frac{1}{16\pi}\left(\int_{\mathbb S^{2}}x{\rm tr}_{h_0}hd{\mu}_{h_0},\int_{\mathbb S^{2}}{\rm tr}_{h_0}hd{\mu}_{h_0}\right),
\end{equation}
where $\mu_{h_0}$ is the area element of $h_0$.
This should be thought of as  an element of $\mathbb R^{3,1}$, the Minkowski space with coordinates $(x,t)\in\mathbb R^3\times\mathbb R$ and
Lorentzian metric
\begin{equation}\label{inner}
ds^2_L=-dt^2+dx^2.
\end{equation}
If $\langle\langle\,,\rangle\rangle$ is the pairing associated with (\ref{inner}), it is proved that
the quantity
\[
\langle\langle\Upsilon_{(M,g)},\Upsilon_{(M,g)}\rangle\rangle
\]
does {\em not} depend on the involved choices and therefore is an invariant of the asymptotic geometry of $(M,g)$; see \cite{Wa}, \cite{CH} and  \cite{H}. If we define the 
{\em future directed light cone} by
\[
\mathcal C^+=\{(x,t)\in \mathbb R^{3,1};t>0, \langle\langle(x,t),(x,t)\rangle\rangle=0\},
\]
and 
denote by $R_g$ the scalar curvature of $(M,g)$ then the corresponding  Positive Mass Theorem can be formulated as follows. 

\begin{theorem}\label{pmt}
\cite{Wa} If $(M,g)$ is an asymptotically hyperbolic $3$-manifold satisfying
$R_g\geq -6$ everywhere then
$\Upsilon_{(M,g)}$ is time-like and future directed unless $(M,g)$ is isometric to hyperbolic $3$-space $\mathbb H^3$ with the standard metric, in which case $\Upsilon_{(M,g)}$ vanishes. Equivalently,
\begin{equation}\label{massineq}
\langle\langle \Upsilon_{(M,g)},\eta\rangle\rangle\leq 0,
\end{equation}
for any $\eta\in\Cc^+$,
with the equality occurring for some $\eta$ if and only if $\Upsilon_{(M,g)}$ vanishes and $(M,g)$ is hyperbolic space.
\end{theorem}

\begin{remark}\label{norm}{\rm
Under the conditions of the theorem, the {\em total mass} of $(M,g)$,
\begin{equation}\label{massdef}
\mathfrak m_{(M,g)}=\sqrt{-\langle\langle \Upsilon_{(M,g)},\Upsilon_{(M,g)}\rangle\rangle},
\end{equation}
is well defined.
Thus, $\mathfrak m_{(M,g)}\geq 0$ and $\mathfrak m_{(M,g)}= 0$ only if $(M,g)$ is hyperbolic space.
Physically, this invariant may be interpreted as the total mass associated to a time-symmetric solution $(\overline M,\overline g)$ of Einstein field equations with a negative cosmological constant $\Lambda<0$ having $(M,g)$ as initial data set. This means that $(M,g)$ embeds as a totally geodesic space-like slice of $(\overline M,\overline g)$, which models an isolated gravitational system. As usual, we assume that  $(\overline M,\overline g)$ satisfies the dominant energy condition along $M$, so  the constraint equations in General Relativity imply that  $R_g\geq 6\Lambda$ \cite{HE}. In Theorem \ref{pmt} and throughout the text we use the normalization $\Lambda=-1$.}
\end{remark}

Theorem \ref{pmt} has been proved by Wang \cite{Wa}, who actually established the result in any dimension $n\geq 3$ assuming that $M$ is spin by using a variant of Witten's method; see also \cite{CH}, where a more general result is proved in the spin category, and \cite{ACG} for a proof of a similar result without the spin assumption in the dimensional range $3\leq n\leq 7$.
We note that elementary proofs for graphs in hyperbolic space have  been obtained recently \cite{dLG1} \cite{dLG2} \cite{DGS}.

The asymptotic expression (\ref{massdef}) for the mass reflects the well-known fact that there exists no meaningful notion of energy density in General Relativity; for  more on this curious aspect of Einstein's theory see \cite{W1} \cite{W2}. In particular, the problem of defining the total mass of a space-like compact domain in a given space-time, and checking that this quasi-local mass has the expected physical properties, is a highly non-trivial matter. Nevertheless, by means of the so-called Hamilton-Jacobi method, Brown and York \cite{BY} were able to define a notion of quasi-local mass in case the bounding surface
can be isometrically embedded in $\mathbb R^3$. A celebrated  result by Shi-Tam \cite{ST1} then guarantees that, at least in the time-symmetric case, the Brown-York quasi-local mass is nonnegative for domains satisfying the dominant energy condition $R_g\geq 0$. Moreover, it vanishes only if the domain lies in $\mathbb R^3$. In complete analogy with this classical picture,
there have been attempts to define a notion of quasi-local mass for compact domains whose boundary can be isometrically embedded in hyperbolic space, so as to obtain (\ref{massvec}) asymptotically at infinity. An interesting proposal in the time-symmetric case has been put forward in works by Wang-Yau \cite{WY1}, Shi-Tam \cite{ST2} and Kwong-Tam \cite{KT}.
To this effect, let $(\Omega,g)$ be a compact $3$-manifold whose boundary $\Sigma$ is (topologically) a sphere whose Gauss curvature  satisfies $K>-1$ everywhere. A result by Pogorelov \cite{P} then implies that $\Sigma$ can be isometrically embedded in hyperbolic space $\mathbb H^3$ and the embedding is unique up to an isometry. Consider $\mathbb H^3\subset\mathbb R^{3,1}$ in the usual manner and let $X:\Sigma\to\mathbb R^{3,1}$ be the position vector of the embedding.
Thus, under these conditions, the {hyperbolic version of the Brown-York quasi-local mass} of $(\Sigma,X)$
is
\begin{equation}\label{massintloc}
\mathfrak m_{BY}(\Sigma,X)=\frac{1}{8\pi}\int_\Sigma \left(H_0-H\right)Xd\Sigma,
\end{equation}
where $H$ is the mean curvature of $\Sigma\subset\Omega$ and $H_0$ is the mean curvature of $\Sigma\subset\mathbb H^3$.
Notice that, as the notation makes it clear,
the right-hand side in (\ref{massintloc}) above depends not only on $\Sigma$ but also on the particular embedding $X$.
It is conjectured that, under the conditions above, if we further assume that $R_g\geq -6$ and $H>0$ then
\begin{equation}\label{conje}
\langle\langle\mathfrak m_{BY}(\Sigma,X),\eta \rangle\rangle\leq 0, \quad \eta\in\Cc^+,
\end{equation}
for any isometric embedding $X:\Sigma\to\mathbb H^3$.
Moreover, equality should hold
for some $\eta$ if and only if
$(\Omega,g)$ is a domain in $\mathbb H^3$; see \cite{ST2}.

\begin{remark}\label{back}
{\rm
To see that (\ref{conje}) is the natural counterpart, in the hyperbolic setting, of the celebrated result by Shi-Tam \cite{ST1} on the positivity of the classical Brown-York quasi-local mass mentioned above,
let us take $\mathbb H^3_{-\kappa^2}$, the hyperbolic space with curvature $-\kappa^2$, $\kappa>0$, as the reference space for the isometric embedding of $\Sigma$. Thus, the time component of (\ref{conje}) is nonnegative in case the conjecture holds with $R_g\geq -6\kappa^2$, which  gives
\[
\int_\Sigma(H_0-H)\cosh \kappa s\,d\Sigma \geq 0,
\]
where $s$ is geodesic distance to the origin.
If we send $\kappa\to 0$ then we get Euclidean space $\mathbb R^3$ in the limit and the inequality becomes
$
\mathfrak m_{BY}(\Sigma) \geq 0
$,
where 
\[
\mathfrak m_{BY}(\Sigma)=\frac{1}{8\pi}\int_{\Sigma}(H_0-H)d\Sigma
\] 
is the standard Brown-York mass of $\Sigma$ \cite{BY}.
This explains the connection between the conjecture and Shi-Tam's main result in \cite{ST1} mentioned above.
We also note the emphasis on the requirement that  (\ref{conje}) should hold for {\em any} isometric embedding $X:\Sigma\to\mathbb H^3$, which reflects the fact that, unlike $\mathfrak m_{BY}(\Sigma)$, $\mathfrak m_{BY}(\Sigma,X)$ depends on the particular embedding of $\Sigma$ in the reference space.
}
\end{remark}

In order to put our results below in their proper context, we now describe a few partial results
in the direction of confirming the conjecture.
We start with the following one, proved in \cite{ST2}, which improves
 \cite[Theorem 1.4]{WY1}.

\begin{theorem}\label{imprst}
\cite{ST2}
Let $(\Omega,g)$ be a compact $3$-manifold whose boundary has positive mean curvature $H>0$. Assume also that $R_g\geq -6$ and that the Gauss curvature of $\Sigma$ is larger than $-1$. Then there exists $\alpha>1$ depending only on the intrinsic geometry of $\Sigma$ such that the vector
\begin{equation}\label{alpha}
\mathfrak m(\Sigma,\alpha)=\int_\Sigma (H-H_0)
X_{\alpha} d\Sigma, \quad X_\alpha=(x,\alpha t)
\end{equation}
satisfies
\begin{equation}\label{alpha2}
\langle\langle \mathfrak m(\Sigma,\alpha),\eta\rangle\rangle\leq 0,
\end{equation}
for any $\eta\in\Cc^+$. Moreover, equality holds for some $\eta$ if and only if $(\Omega,g)$ is a domain in $\mathbb H^3$.
\end{theorem}

\begin{remark}\label{stalpha1}
{\rm It is shown in \cite{ST2} that
\[
\alpha=\coth R_1+\frac{1}{\sinh R_1}\sqrt{\frac{\sinh^2R_2}{\sinh^2R_1}-1},
\]
where the image of $\Sigma$ under the embedding is supposed to lie between geodesic spheres of radius $R_1<R_2$ centered at the origin. Thus, $\alpha\to 1$ if $R_1,R_2\to +\infty$ in such a way that $2R_1-R_2\to +\infty$.
We also note that Kwong \cite{K} has generalized Theorem \ref{imprst} to any dimension $n\geq 3$.
}
\end{remark}

We now turn to a result by Kwong and Tam \cite{KT}. For this we need to introduce some further 
notation. For $\epsilon>0$ small enough we set $\Sigma_{\epsilon}=\rho^{-1}(\epsilon)$, where $\rho$ is the fixed defining function. For further reference, we then say that $(\Omega_\epsilon,\Sigma_\epsilon)$ is a {\em coordinate exhaustion} of $(M,g)$ if $\Omega_\epsilon$ is the compact domain such that $\partial\Omega_\epsilon=\Sigma_\epsilon$. It is proved in \cite{KT}
that the Gauss curvature of $\Sigma_\epsilon$ with the induced metric satisfies
\[
K=\sinh^2\epsilon +O(\epsilon^5).
\]
Thus, by Pogorelov's result mentioned above, there exists  an isometric embedding $\Sigma_\epsilon\subset \mathbb H^3$ for all $\epsilon>0$ small enough. The main result in \cite{KT} says that by suitably composing this embedding with an isometry, the resulting Brown-York quasi-local mass vector converges to Wang's mass $\Upsilon_{(M,g)}$ as $\epsilon\to 0$.

\begin{theorem}\label{ktmain}\cite{KT}
Let $(M,g)$ be an asymptotically hyperbolic manifold. Then for all $\epsilon>0$ sufficiently small there exists an isometric embedding $X^{(\epsilon)}:\Sigma_{\epsilon}\to \mathbb H^3$ such that
\begin{equation}\label{ktmaineq}
\lim_{\epsilon\to 0}\mathfrak m_{BY}(\Sigma_\epsilon,X^{(\epsilon)})=\Upsilon_{(M,g)}.
\end{equation}
\end{theorem}

As already noticed in \cite{KT}, if we combine this with Theorem \ref{pmt} it follows that  $\mathfrak m_{BY}(\Sigma_\epsilon,X^{(\epsilon)})$ is time-like and future directed for all $\epsilon>0$ small enough, whenever $(M,g)$ is not isometric to hyperbolic space. 
The main purpose of this note is to show that this also holds true for {\em any} choice of isometric embeddings $X_\epsilon:\Sigma_\epsilon\to \mathbb H^3$; see Theorem \ref{mmain2} below. In particular, this confirms that in a sense the hyperbolic Brown-York mass of a sufficiently large domain has the conjectured causal character.

\begin{theorem}\label{mmain1}
If $(\Omega_\epsilon,\Sigma_\epsilon)$ is any coordinate exhaustion of an asymptotically hyperbolic  $3$-manifold $(M,g)$ satisfying $R_g\geq -6$ then 
\begin{equation}\label{limit}
\langle\langle\lim_{\epsilon\to 0}{\mathfrak m}_{BY}(\Sigma_\epsilon,X_\epsilon),\eta\rangle\rangle\leq 0, \quad \eta\in\Cc^+,
\end{equation}
for any choice of isometric embeddings $X_\epsilon:\Sigma_\epsilon\to\mathbb H^3$. Moreover, the equality holds for some $\eta\in\Cc^+$ only if $(M,g)$ is isometric to hyperbolic space.
\end{theorem}

This immediately yields the following result; compare with \cite[Corollary 1.1]{KT}.

\begin{theorem}\label{mmain2}
Let $(\Omega_\epsilon,\Sigma_\epsilon)$ be a coordinate exhaustion of an asymptotically hyperbolic  $3$-manifold $(M,g)$ satisfying $R_g\geq -6$. Assume that $(M,g)$ is not isometric to $(\mathbb H^3,g_0)$. Then, for any given choice of isometric embeddings $X_\epsilon:\Sigma_\epsilon\to\mathbb H^3$,
the vector
$
{\mathfrak m}_{BY}(\Sigma_\epsilon, X_\epsilon)
$
is time-like and future directed
for all $\epsilon>0$ small enough.
\end{theorem}


The results above actually follow from a Witten-type formula for the limit in (\ref{limit}). As explained in Section \ref{proofmain1}, to each $\eta\in\Cc^+$ we can attach an imaginary Killing spinor $\phi^{(\eta)}$  on $\mathbb H^3$. By carefully identifying the two infinities and adapting Witten's method in the standard way \cite{AD} \cite{Wa} \cite{CH} \cite{WY1} we will be able to find a Killing-harmonic spinor $\psi^{(\eta)}$ on $M$ which asymptotes $\phi^{(\eta)}$ at infinity in a suitable manner.     
The above mentioned Witten-type formula is expressed in terms of $\psi^{(\eta)}$ as follows.

\begin{theorem}
\label{mmain3}
If $(M,g)$ is an asymptotically hyperbolic manifold  satisfying $R_g\geq -6$ then  
\begin{equation}\label{newproof2}
\int_M\left(|\widetilde \nabla^+\psi^{(\eta)}|^2+\frac{R_g+6}{4}|\psi^{(\eta)}|^2\right)dM
=-4\pi\langle\langle \lim_{\epsilon\to 0}\mathfrak m_{BY}(\Sigma_\epsilon,X_\epsilon),\eta\rangle\rangle,
\end{equation}
for any $\eta\in\Cc^+$ and any choice of isometric embedding $X_\epsilon:\Sigma_\epsilon\to \mathbb H^3$.
\end{theorem}

Here, $\widetilde{\nabla}^+$ is the Killing connection acting on spinors; see (\ref{killingcon}).

\begin{remark}
\label{corolll}
{\rm 
If we take $X_\epsilon=X^{(\epsilon)}$, the Kwong-Tam normalized embeddings appearing in Theorem \ref{ktmain}, then (\ref{newproof2}) yields \begin{equation}\label{newproof22}
\int_M\left(|\widetilde \nabla^+\psi^{(\eta)}|^2+\frac{R_g+6}{4}|\psi^{(\eta)}|^2\right)dM
=-4\pi\langle\langle \Upsilon_{(M,g)},\eta\rangle\rangle,
\end{equation}
which is precisely the mass formula appearing in \cite{Wa}; see also \cite{CH}. 
Moreover, in the spirit of Remark \ref{back}, we might consider time components and then send $\kappa\to 0$. Clearly, the limiting $3$-manifold $(M,g)$ is asymptotically flat and we will eventually obtain
\begin{equation}\label{neproof3}
\int_M\left(|\nabla\psi|^2+\frac{R_g}{4}|\psi|^2\right)dM=4\pi\lim_{r\to +\infty}\mathfrak m_{BY}(\Sigma_r),
\end{equation}
where $r$ is the asymptotic parameter labeling the coordinate spheres and  $\psi$ is a harmonic spinor on $M$ which approaches a parallel spinor of unit norm at infinity. Since it is well-known \cite{BY} \cite{ST1} that the limit in the right-hand side equals  the ADM mass of $(M,g)$ we thus obtain Witten's celebrated formula for this mass \cite{Wi}.
}
\end{remark}

The results above are obtained by exploring the full power of Witten's method, which consists of finding a suitable spinor globally defined on $M$. If one is interested, however, in merely obtaining the inequality in (\ref{limit}), this can be accomplished by appealing  to the theory of boundary value problems for the Dirac operators on {\em compact} domains. This approach is inspired on recent work by Hijazi, Montiel and Raulot \cite{HM} \cite{HMRa}.
We thus consider the vector-valued quantity
\begin{equation}\label{noss}
\widehat{\mathfrak m}(\Sigma,X)=\frac{1}{8\pi}\int_{\Sigma}\frac{H_0^2-H^2}{H+2}Xd\Sigma,
\end{equation}
where we assume that $\Sigma$ is the (spherical) boundary of a compact $3$-manifold with mean curvature $H>-2$, Gauss curvature $K>-1$ and $X:\Sigma\to \mathbb R^{3,1}$ is the position vector of an isometric embedding $\Sigma\subset \mathbb H^3$ whose mean curvature is $H_0$.

\begin{theorem}\label{main1}
Under these conditions, if we assume further that $H$ is constant then
\begin{equation}\label{main12}
\langle\langle \widehat{\mathfrak m}(\Sigma, X),\eta\rangle\rangle\leq 0, \quad \eta \in\Cc^+,
\end{equation}
for any isometric embedding $X:\Sigma\to\mathbb H^3$.
Moreover, equality holds if and only if $\Omega\subset \mathbb H^3$ is a round ball.
\end{theorem}

A result 
by Neves and Tian \cite{NT} says that any asymptotically hyperbolic manifold satisfying ${\rm tr}_{h_0}h>0$
has the property that a neighborhood of infinity can be uniquely foliated by stable constant mean curvature surfaces. Moreover, a recent result by Ambrozio \cite{Am} implies that for any sufficiently small perturbation of an anti-de-Sitter-Schwarzschild space of positive mass the Neves-Tian foliation can be extend to a {\em global} foliation by constant mean curvature surfaces. This provides examples of domains for which Theorem \ref{main1} and its generalization described in Remark \ref{lamb} below apply.  
But notice that, if compared with the main results in \cite{HM} \cite{HMRa}, $\widehat{\mathfrak m}(\Sigma,X)$ has an obvious drawback, namely, it only has the expected causal character under the rather stringent constant mean curvature condition on $\Sigma\subset\Omega$.
Nevertheless, our next results show not only that
$\widehat{\mathfrak m}(\Sigma,X)$ has very nice causal properties when evaluated on the asymptotic limit (Theorem \ref{main2}) but also that at infinity it captures the corresponding limit of the hyperbolic Brown-York mass (Theorem \ref{main3}). 

\begin{theorem}\label{main2}
If $(\Omega_\epsilon,\Sigma_\epsilon)$ is any coordinate exhaustion of an asymptotically hyperbolic  $3$-manifold $(M,g)$ satisfying $R_g\geq -6$ then \begin{equation}\label{limit2}
\langle\langle\lim_{\epsilon\to 0}\widehat{\mathfrak m}(\Sigma_\epsilon,X_\epsilon),\eta\rangle\rangle\leq 0, \quad \eta\in\Cc^+,
\end{equation}
for any choice of isometric embeddings $X_\epsilon:\Sigma_\epsilon\to\mathbb H^3$.
\end{theorem}

\begin{theorem}\label{main3}
If $(\Omega_\epsilon,\Sigma_\epsilon)$ is any coordinate exhaustion of an asymptotically hyperbolic $3$-manifold $(M,g)$ then \begin{equation}\label{limit3}
\lim_{\epsilon\to 0}\widehat{\mathfrak m}(\Sigma_\epsilon,X_\epsilon)=\lim_{\epsilon\to 0}{\mathfrak m}_{BY}(\Sigma_\epsilon,X_\epsilon),
\end{equation}
for any choice of isometric embeddings $X_\epsilon:\Sigma_\epsilon\to\mathbb H^3$. In particular, the mass inequality (\ref{limit}) holds under the conditions of Theorem \ref{mmain1}. 
\end{theorem} 

Notice that this approach does not give a proof of the rigidity statement in Theorem \ref{mmain2}. For this we need to apply alternative methods; see \cite{BaW} or Theorem \ref{mmain1} above.


\begin{remark}\label{wangyau}
{\rm Examples due to \'O Murchadha-Szabados-Tod \cite{OST} show that the classical Brown-York mass \cite{BY} and its generalization by Kijowski \cite{Ki} and Liu-Yau \cite{LY}, which use $\mathbb R^3$ as the reference space for the isometric embedding of the boundary, might be strictly positive even when the surface lies in $\mathbb R^{3,1}$, a result that clearly contradicts physical intuition. As confirmed by recent breakthroughs by Chen, Wang and Yau \cite{WY2} \cite{WY3} \cite{CWY}, it turns out that a much more satisfactory definition of quasi-local mass should use the full space-time $\mathbb R^{3,1}$ as reference space. Besides having many other interesting properties, this new quasi-local mass vanishes for any admissible surface lying in
$\mathbb R^{3,1}$. We would like to point out, however, that in the special but important time-symmetric case treated in this paper,  the Brown-York mass and its hyperbolic version (\ref{massintloc}) still provide rigidity results consistent with physical expectation, as illustrated by the various results described in this Introduction.}
\end{remark}

This paper is organized as follows. In Section \ref{dirac} and Appendix \ref{appendixA} we review the results on Dirac operators used throughout the text. This is applied to prove a sort of holographic principle for imaginary Killing spinors on domains whose scalar curvature is bounded from below by a negative constant (Proposition \ref{estimat}), which is the main ingredient in the proofs of Theorems \ref{main1}, \ref{main2} and \ref{main3} in
Section \ref{proofrest}.
The proofs of Theorems \ref{mmain1}, \ref{mmain2} and \ref{mmain3} also makes use of this spin machinery and are included in Section \ref{proofmain1}.


\section{Dirac operators on $3$-manifolds with boundary}\label{dirac}

In this section we review the results in the theory of Dirac operators on  $3$-manifolds needed in the rest of the paper.  The reader will find more detailed presentations of this preparatory  material in \cite{BFGK}, \cite{F},  \cite{HMR1}, \cite{HMR2}, \cite{HM} and \cite{HMZ}.
Even though much of the discussion below holds in any dimension, we will restrict ourselves to the physically relevant case $n=3$.

We consider an {\em orientable} $3$-manifold $\Omega$ endowed with a Riemannian metric $g$. It is well-known that such a manifold is automatically {spin}, which allows us to  fix once and for all a spin structure on  $T\Omega$. We denote by $\mathbb S\Omega$ the associated spin bundle and by $\nabla$ both the Levi-Civita connection of $T\Omega$
and its lift to $\mathbb S\Omega$.

If $\gamma:T\Omega\times\mathbb S\Omega\to\mathbb S\Omega$ is the Clifford product,
we define the {\em Killing connections} acting on a spinor $\psi\in\Gamma(\mathbb SM)$ by
\begin{equation}\label{killingcon}
\widetilde\nabla_X^\pm\psi=\nabla_X\psi\pm\frac{i}{2}\gamma(X)\psi,
\end{equation}
so that a spinor is parallel with respect to $\widetilde \nabla$ if and only if, by definition, it is an {\em imaginary Killing spinor}; see \cite{BFGK}. The corresponding {\em Killing-Dirac operators} are defined in the standard way, namely,
\begin{equation}\label{killingdirac}
\widetilde D^\pm\psi=\sum_{i=1}^3\gamma(e_i)\widetilde\nabla_{e_i}^\pm\psi,
\end{equation}
so that
\begin{equation}\label{killdirexp}
\widetilde D^\pm=D\mp\frac{3i}{2},
\end{equation}
where 
\[
D\psi=\sum_{i=1}^3\gamma(e_i)\nabla_{e_i}\psi
\]
is the standard Dirac operator. We say that a spinor $\psi$ is {\em Killing-harmonic} if it satisfies any of the linear equations $\widetilde D^{\pm}\psi=0$. 


Given a spinor $\psi$ we  set
\[
\widetilde\omega^\pm(X)=-\langle \widetilde{\Wc}^\pm(X)\psi,\psi\rangle, \quad X\in \Gamma(T\Omega),
\]
where 
\[
\widetilde \Wc^{\pm}(X)=-(\widetilde\nabla_X^\pm+\gamma(X) \widetilde D^\pm).
\]
We easily compute that
\[
{\rm div}\,\widetilde\omega^\pm=|\widetilde\nabla^\pm \psi|^2-|\widetilde D^\pm\psi|^2+\frac{R_g+6}{4}|\psi|^2,
\]
so if $\Omega$
is compact with a nonempty boundary $\Sigma$, which we assume oriented by its inward pointing unit normal $\nu$, then integration by parts yields the integral version of the fundamental Lichnerowicz formula, namely, 
\begin{equation}\label{partshyp}
\int_\Omega\left(|\widetilde\nabla^\pm \psi|^2-|\widetilde D^\pm\psi|^2+\frac{R_g+6}{4}|\psi|^2\right)d\Omega={\rm Re}\int_{\Sigma}
\left\langle \widetilde\Wc^\pm(\nu)\psi,\psi\right\rangle d\Sigma.
\end{equation}

A key step in our argument is to rewrite the right-hand side of (\ref{partshyp}) in terms of the geometry of $\Sigma$. First note that
$\Sigma$ carries the spin bundle $\mathbb S\Omega|_\Sigma$, obtained by restricting  $\mathbb S\Omega$ to $\Sigma$. This becomes a Dirac bundle if its Clifford product is
$$
\gamma^{\intercal}(X)\psi=\gamma(X)\gamma(\nu) \psi,\quad X\in \Gamma(T\Sigma), \quad \psi\in \Gamma(\mathbb S\Omega|_\Sigma),
$$
and its connection is
\begin{equation}\label{conn0}
\nabla^{\intercal}_X\psi  =  \nabla_X\psi-\frac{1}{2}\gamma^{\intercal}(AX)\psi,
\end{equation}
where $A=-\nabla\nu$ is the shape operator of $\Sigma$.
The corresponding Dirac operator $D^{\intercal}:\Gamma(\mathbb S\Omega|_\Sigma)\to\Gamma(\mathbb S\Omega|_\Sigma)$ is
$$
D^{\intercal}\psi=\sum_{j=1}^{3}\gamma^{\intercal}(f_j)\nabla^{\intercal}_{f_j}\psi,
$$
where $\{f_j\}_{j=1}^{2}$ is a local orthonormal tangent frame to $\Sigma$.
Imposing that $Af_j=\kappa_jf_j$, where $\kappa_j$ are the principal curvatures of $\Sigma$, we have
\begin{equation}\label{form1}
D^{\intercal}\psi=-\gamma(\nu)\Dc\psi+\frac{H}{2}\psi,
\end{equation}
where $H=\kappa_1+\kappa_2$ is the mean curvature and
\begin{equation}\label{form2}
\Dc=\gamma(\nu)\left(D^{\intercal}-\frac{H}{2}\right)=
\sum_{j=1}^{2}\gamma(f_j)\nabla_{f_j}\psi.
\end{equation}
Since $\Dc=D-\gamma(\nu)\nabla_{\nu}$, we obtain
\begin{equation}\label{par}
D^{\intercal}\psi=\frac{H}{2}\psi-(\nabla_\nu+\gamma(\nu)D)\psi.
\end{equation}

We now observe that since the unit normal field $\nu$ provides an orientation for the normal bundle of $\Sigma$,
$T\Sigma$ carries a preferred spin structure with spin bundle $\mathbb S\Sigma$.
As usual we denote by $\nabla^{\Sigma}$ both the Levi-Civita connection on $T\Sigma$ and its lift to $\mathbb S\Sigma$. In particular, $\Sigma$ has an {\em intrinsic} Dirac operator $D^\Sigma:\Gamma(\mathbb S\Sigma)\to\Gamma(\mathbb S\Sigma)$
defined by
$$
D^{\Sigma}\varphi=\sum_{j=1}^{2}\gamma^{\Sigma}(f_j) \nabla^\Sigma_{f_j}\varphi,
$$
where $\gamma^\Sigma:T\Sigma\times \mathbb S\Sigma\to \mathbb S\Sigma$ is the Clifford product.
It turns out that the embedding $\Sigma\subset\Omega$ induces natural identifications $\mathbb S\Omega|_\Sigma=\mathbb S\Sigma$, $\nabla^\intercal=\nabla^\Sigma$ and  $\gamma^\intercal=\gamma^\Sigma$, so that
\begin{equation}\label{ident}
D^\intercal=D^\Sigma,
\end{equation}
and by (\ref{par}),
(\ref{partshyp}) becomes
\begin{equation}\label{parts3}
\int_\Omega\left(|\widetilde\nabla^{\pm} \psi|^2-|\widetilde D^\pm\psi|^2+\frac{R_g+6}{4}|\psi|^2\right)d\Omega=\int_{\Sigma}
\left(\langle \widetilde D^{\Sigma}_\pm\psi,\psi\rangle-\frac{H}{2}|\psi|^2\right) d\Sigma,
\end{equation}
where
\begin{equation}\label{newdirac}
\widetilde D^{\Sigma}_\pm=D^{\Sigma}\pm i\gamma(\nu):\Gamma(\mathbb S{\Sigma})\to
\Gamma(\mathbb S{\Sigma}).
\end{equation}
By Cauchy-Schwarz, we have $|\widetilde D^{\pm}\psi|^2\leq 3|\widetilde\nabla^\pm\psi|^2$, so that
\begin{equation}\label{parts4}
\int_{\Sigma}
\left(\langle \widetilde D^{\Sigma}_\pm\psi,\psi\rangle-\frac{H}{2}|\psi|^2\right) d\Sigma\geq
\int_\Omega\left(-\frac{2}{3}|\widetilde D^\pm\psi|^2+\frac{R_g+6}{4}|\psi|^2\right)d\Omega,
\end{equation}
with the equality occurring if and only if $\psi$ is a {\em twistor spinor}, i.e.
\begin{equation}\label{twistor}
\nabla_X\psi+\frac{1}{3}\gamma(X)D\psi=0.
\end{equation}

From this we immediately obtain an useful integral inequality for compact domains whose scalar curvature is bounded from below by a negative constant.

\begin{proposition}\label{montneg}
\cite{HMR2}
If  $(\Omega,g)$ satisfies $R_g\geq -6$ then
\begin{equation}\label{reillyhyp2}
\int_{\Sigma}
\left(\langle \widetilde D^{\Sigma}_\pm\psi,\psi\rangle-\frac{H}{2}|\psi|^2\right)d\Sigma\geq -\frac{2}{3}\int_\Omega\left|\widetilde D^\pm\psi\right|^2d\Omega,
\end{equation}
with the  equality occurring  if and only if $\psi$ is a twistor spinor and $R_g\equiv -6$.
\end{proposition}

We use this for a Killing-harmonic spinor obtained by solving a suitable boundary value problem.
To explain this
we observe that $i\gamma(\nu):\mathbb S\Sigma\to \mathbb S\Sigma$ is a ({pointwise}) self-adjoint involution. We thus consider the corresponding projection operators
\[
P_{\pm}=\frac{1}{2}\left({\rm Id}_{\mathbb S\Sigma}\pm i\gamma(\nu)\right)
\]
onto the eigenbundles of $i\gamma(\nu)$.
For any $\phi\in\Gamma(\mathbb S\Sigma)$
we set
$\phi_\pm=P_\pm\phi$,
so that
\[
\phi=\phi_++\phi_-,
\]
a pointwise orthogonal decomposition. Also, since $D^\intercal\gamma(\nu)=-\gamma(\nu)D^\intercal$ we have
$D^\Sigma P_\pm=P_\mp D^\Sigma$ and hence
\begin{equation}\label{cancel1}
\langle D^\Sigma \phi,\phi\rangle=\langle D^\Sigma \phi_+,\phi_-\rangle+\langle D^\Sigma \phi_-,\phi_+\rangle.
\end{equation}
In particular, 
\begin{equation}
\label{cancell}
\langle D^\Sigma \phi_+,\phi_+\rangle=\langle D^\Sigma \phi_-,\phi_-\rangle=0.
\end{equation}
Upon integration of (\ref{cancel1}) we find that
\begin{equation}\label{curc}
\int_\Sigma\langle D^\Sigma\phi,\phi\rangle d\Sigma=2\int_\Sigma\langle D^\Sigma\phi_+,\phi_-\rangle d\Sigma.
\end{equation}
Finally, it is easy to see that
\begin{equation}\label{curc2}
|D^\Sigma \phi|^2=|D^\Sigma \phi_+|^2+|D^\Sigma \phi_-|^2.
\end{equation}

It turns out that the projections $P_\pm$ define local elliptic boundary conditions for $\widetilde D^\pm$, as the following result shows. These are the so-called {\em MIT bag boundary conditions}.

\begin{proposition}\label{bvpmr}
Let $(\Omega,g)$ be as above with $R_g\geq -6$ and $H> -2$. Then, for any $\phi\in\Gamma(\mathbb S{\Sigma})$, the boundary value problem
\[
\left\{
\begin{array}{rclc}
\widetilde D^\pm\psi^\pm & = &  0 & {\rm on}\,\Omega \\
\psi^\pm_\pm & = &  \phi_{\pm} & {\rm on}\,\Sigma
\end{array}
\right.
\]
has a unique smooth solution $\psi^{\pm}\in\Gamma(\mathbb S\Omega)$.
\end{proposition}

A proof of this proposition can be found in A. Rold\'an's doctoral thesis \cite{R}.
For convenience we reproduce the argument in the Appendix.

We now show how this theory yields a nice inequality for arbitrary spinors on $\Sigma$ involving the mean curvature. This is a kind of holographic principle for imaginary Killing spinors and is obviously inspired on  \cite[Proposition 9]{HM}; see also \cite{HMRa}.

\begin{proposition}\label{estimat}
Let $(\Omega,g)$ be a compact oriented (and hence spin)  $3$-manifold with boundary $\Sigma$. Assume that $R_g\geq -6$ and $H>-2$. Then there holds
\begin{equation}\label{estimateq}
\int_\Sigma\left(\frac{|D^\Sigma\phi|^2}{H+2}-\frac{H-2}{4}|\phi|^2\right)d\Sigma\geq 0,
\end{equation}
for any $\phi\in\Gamma(\mathbb S\Sigma)$. Moreover, equality holds if and only if there exist imaginary Killing spinors $\psi^\pm\in\Gamma(\mathbb S\Omega)$ such that $\psi^\pm_\pm=\phi_\pm$ along $\Sigma$.
\end{proposition}

\begin{proof}
Given  $\phi\in\Gamma(\mathbb S\Omega|_{\Sigma})$, we take $\psi^+\in\Gamma(\mathbb S\Omega)$ as in  Proposition \ref{bvpmr}.
It follows from (\ref{reillyhyp2}) that
\begin{equation}\label{preli}
\int_\Sigma\left(\langle \widetilde D^\Sigma_- \psi^+,\psi^+\rangle-\frac{H}{2}|\psi^+|^2\right)d\Sigma\geq 0.
\end{equation}
Now,
\[
\langle \widetilde D^\Sigma\psi^+,\psi^+\rangle=\langle D^\Sigma\psi^+,\psi^+\rangle+|\psi^+_+|^2-|\psi^+_-|^2,
\]
so that (\ref{curc}) gives
\begin{equation}\label{first}
\int_\Sigma\left(2 \langle D^\Sigma\psi^+_+,\psi^+_-\rangle-\left(\frac{H}{2}-1\right)|\psi^+_+|^2-\left(\frac{H}{2}+1\right)
|\psi^+_-|^2\right)d\Sigma\geq 0.
\end{equation}
Since $H>-2$ we can use the elementary estimate
\begin{eqnarray*}
0 & \leq & \left|\frac{D^\Sigma\psi^+_+}{\sqrt{\frac{H}{2}+1}}-\sqrt{\frac{H}{2}+1}\psi^+_-\right|^2\\
 & = & \frac{|D^\Sigma\psi^+_+|^2}{\frac{H}{2}+1}-2\langle D^\Sigma\psi^+_+,\psi^+_-\rangle+\left(\frac{H}{2}+1\right)|\psi^+_-|^2,
\end{eqnarray*}
and leading this to (\ref{first})
we obtain
\begin{equation}\label{first2}
\int_\Sigma\left(\frac{|D^\Sigma\psi^+_+|^2}{H+2}-\frac{H-2}{4}|\psi^+_+|^2\right)d\Sigma\geq 0.
\end{equation}

On the other hand, if $\phi\in\Gamma(\mathbb S\Omega|_\Sigma)$ is the {\em same} spinor as above, we take $\psi^-\in\Gamma(\mathbb S\Omega)$ as in Proposition \ref{bvpmr}.
It follows from (\ref{reillyhyp2}) that
\[
\int_\Sigma\left(\langle \tilde D^\Sigma_+ \psi^-,\psi^-\rangle-\frac{H}{2}|\psi^-|^2\right)d\Sigma\geq 0.
\]
We thus get the analogue of (\ref{first}), namely,
\begin{equation}\label{second}
\int_\Sigma\left(2 \langle D^\Sigma\psi^-_+,\psi^-_-\rangle-\left(\frac{H}{2}-1\right)|\psi^-_-|^2-
\left(\frac{H}{2}+1\right)|\psi^-_+|^2\right)d\Sigma\geq 0.
\end{equation}
We can again estimate
\begin{eqnarray*}
0 & \leq & \left|\frac{D^\Sigma\psi^-_-}{\sqrt{\frac{H}{2}+1}}-\sqrt{\frac{H}{2}+1}\psi^-_+\right|^2\\
 & = & \frac{|D^\Sigma\psi^-_-|^2}{\frac{H}{2}+1}-2\langle D^\Sigma\psi^-_-,\psi^-_+\rangle+\left(\frac{H}{2}+1\right)|\psi^-_+|^2,
\end{eqnarray*}
and leading this to (\ref{second})
we obtain
\begin{equation}\label{second2}
\int_\Sigma\left(\frac{|D^\Sigma\psi^-_-|^2}{H+2}-\frac{H-2}{4}|\psi^-_-|^2\right)d\Sigma\geq 0.
\end{equation}
Adding (\ref{first2}) to (\ref{second2}) and using (\ref{curc2}) and the boundary conditions, (\ref{estimateq}) follows.

If equality holds in (\ref{first2}) then it also holds in (\ref{preli}). By Proposition \ref{montneg}, the Killing-harmonic spinor $\psi^+$ is a twistor spinor as well. Thus, $\psi^+$ is an imaginary Killing spinor, as desired. Since an entirely similar argument handles the case of equality in (\ref{second2}), this completes the proof of the proposition.
\end{proof}

\section{The proofs of Theorems \ref{mmain1}, \ref{mmain2} and \ref{mmain3}}\label{proofmain1}

For simplicity of notation in this section we write $\widetilde{\nabla}=\widetilde\nabla^+$ for the Killing connection and consistently drop the plus sign in the rest of the notation.  

We start 
the proof of 
Theorems \ref{mmain1}, \ref{mmain2} and \ref{mmain3}
by recalling
that $\hat\phi\in\Gamma(\mathbb S\mathbb H^3)$ is called an {\em imaginary Killing spinor} if it satisfies
\begin{equation}\label{kills}
\nabla^0_X\hat\phi+\frac{i}{2}\gamma^0(X)\hat\phi=0,\quad X\in\Gamma(T\mathbb H^3),
\end{equation}
where the  label $0$  refers to invariants associated to $\mathbb H^3$ and its spin bundle.

\begin{lemma}\label{baum}
If $\hat\phi\in\Gamma(\mathbb S\mathbb H^3)$ is an imaginary Killing spinor then
\begin{equation}\label{baum1}
{\rm Hess}^0|\hat\phi|^2=|\hat\phi|^2g_0
\end{equation}
and
\begin{equation}\label{baum2}
|\nabla^0|\hat\phi|^2|^2=-\sum_{i=1}^3\langle \gamma^0(\tilde e_i)\hat\phi,\hat\phi\rangle^2,
\end{equation}
where $\{\tilde e_i\}_{i=1}^3$ is a local orthonormal tangent frame. Moreover, if $\gamma:[0,+\infty)\to\mathbb H^3$ is  a normal geodesic then $u=|\hat\phi|^2\circ\gamma$ is given by
\begin{equation}\label{baum3}
u(t)=Ae^{t}+Be^{-t}, \quad A,B\in\mathbb R.
\end{equation}

\begin{proof}
If $X,Y\in\Gamma(T\mathbb H^3)$ then direct computations give
\begin{equation}\label{baum4}
X|\hat\phi|^2=-i\langle\gamma^0(X)\hat\phi,\hat\phi\rangle,
\end{equation}
and
\begin{equation}\label{baum5}
XY|\hat\phi|^2=-i\langle\gamma^0(\nabla^0_XY)\hat\phi,\hat\phi\rangle+\langle X,Y\rangle|\hat\phi|^2.
\end{equation}
From this, (\ref{baum1}) and (\ref{baum2}) follow easily. It also follows from (\ref{baum5}) that $u''=u$, which proves (\ref{baum3}).
\end{proof}
\end{lemma}

For further reference we note that  if $h_\epsilon=g|_{\Sigma_\epsilon}$ then (\ref{expan}) and (\ref{asymexpcoll}) give
\begin{equation}\label{meteps}
h_\epsilon=\epsilon^{-2}h_0+O(1),
\end{equation}
so that
\begin{equation}\label{areagrowth}
d\Sigma_\epsilon=(\epsilon^{-2}+O(1))d\mu_{h_0}.
\end{equation}
Moreover, 
(\ref{baum3}) implies
\begin{equation}\label{partic}
|\hat\phi|^2= O(\epsilon^{-1}),
\end{equation}

We also need an explicit description of imaginary Killing spinors on $\mathbb H^3$. 
With respect to a suitable trivialization $\mathbb S\mathbb H^3=\mathbb H^3\times\mathbb C^2$, the expression
\[
\phi^{(z)}=
\left(
\begin{array}{c}
\left(z_1e^{\frac{i}{2}\vartheta}\cos\frac{\theta}{2}+
z_2e^{-\frac{i}{2}\vartheta}\sin\frac{\theta}{2}\right)e^{\frac{r}{2}}\\
-\left(z_1e^{\frac{i}{2}\vartheta}\sin\frac{\theta}{2}-
z_2e^{-\frac{i}{2}\vartheta}\cos\frac{\theta}{2}\right)e^{-\frac{r}{2}}
\end{array}
\right), \quad z=(z_1,z_2)\in\mathbb C^2,
\]
defines an imaginary Killing spinors in $\mathbb H^3$, where
$X=(r,\theta,\vartheta)$ is the position vector expressed in standard polar coordinates \cite{WY1} \cite{Z}. A straightforward computation gives
\begin{equation}\label{gives}
|\phi^{(z)}|^2=-\langle\langle X,\overline \eta(z)\rangle\rangle,
\end{equation}
where
\[
\overline \eta(z)=(-(|z_1|^2-|z_2|^2), -(z_1\overline z_2+\overline z_1z_2),-i(z_1\overline z_2-\overline z_1 z_2), |z_1|^2+|z_2|^2).
\]
It is easy to see that $\overline \eta$ maps $\mathbb S^{3}=\{(z\in\mathbb C^2;|z|=1\}$ onto $\mathbb S^2=\Cc^+\cap \{t=1\}$. In fact, $\overline \eta|_{\mathbb S^3}$ is just the Hopf map.
Thus, $\overline \eta$  maps $\mathbb C^2$ onto $\Cc^+$.

With these preliminaries at hand, we finally start the proof of Theorem \ref{mmain3}.
For $\eta\in \Cc^+$  we choose $z\in\mathbb C^2$ so that $\overline \eta(z)=\eta$. The usual adaptation of Witten's method \cite{AD} \cite{CH} \cite{Wa} \cite{WY1} leads to the existence of a non-trivial Killing-harmonic spinor $\psi^{(\eta)}\in\Gamma(\mathbb SM)$, $\widetilde D^+\psi^{(\eta)}=0$, which asymptotes the imaginary Killing spinor $\phi^{(z)}\in\Gamma(\mathbb SH^3)$ after a suitable identification of the two infinities.
More precisely, in a neighborhood $U$ of the conformal infinity we consider the background hyperbolic metric
\[
g_0=\sinh^{-2}\rho(d\rho^2+h_0),
\]
and the gauge transformation $\Bc$  given by
\[
g(\Bc X,\Bc Y)=g_0(X,Y), \quad g(\Bc X,Y)=g(X,\Bc Y), \quad X,Y\in\Gamma(TU).
\]
We find from (\ref{remainder}) that
\begin{equation}\label{gauge}
|\Bc-I|+|\nabla^0 \Bc|=O(\epsilon^3).
\end{equation}
Also, $\Bc$ defines a fiberwise isometry  between the spin bundles on $U$
endowed with the metric structures coming from $g_0$ and $g$. Thus, if $f$ is a cut-off function on $M$ with $f\equiv 1$ in a neighborhood of infinity and $\phi$ is a spinor on $U$, $f\Bc \phi$ defines a spinor on $M$, denoted by $\phi_*$.
If we apply this construction to $\phi=\phi^{(z)}$ then a well-known computation
gives
\[
|\widetilde\nabla\phi_*^{(z)}|\leq C\left(|\Bc^{-1}||\Bc-I|+|\nabla^0 \Bc|\right)|\phi_*^{(z)}|=O(\epsilon^{5/2}),
\]
so that 
$\widetilde \nabla\phi_*^{(z)}\in L^2(\mathbb SM)$ and hence  $\widetilde D\phi_*^{(z)}\in L^2(\mathbb SM)$.
A standard argument then implies the existence of a spinor $\xi\in H^1(\mathbb S M)$ satisfying
$
\widetilde D\xi=-\widetilde D\phi_*^{(z)}
$.
Thus, $\psi^{(\eta)}=\phi_*^{(z)}+\xi$ is Killing-harmonic and asymptotes $\phi_*^{(z)}$ at infinity in the sense that $\psi^{(\eta)}-\phi_*^{(z)}\in L^2(\mathbb SM)$.
With this spinor at hand we may apply (\ref{parts3}) to obtain
\begin{eqnarray*}
\int_M\left(|\widetilde\nabla\psi^{(\eta)}|^2
        +\frac{R_g+6}{4}|\psi^{(\eta)}|^2\right)dM
    & {=} & \lim_{\epsilon\to 0}
    \int_{\Sigma_\epsilon}\left(\langle \widetilde D^{\Sigma_\epsilon}\psi,\psi\rangle-\frac{H}{2}|\psi|^2\right)d\Sigma_\epsilon\\
&=& \lim_{\epsilon\to 0}\int_{\Sigma_\epsilon}\left(\langle \widetilde D^{\Sigma_\epsilon}\phi_*^{(z)},\phi_*^{(z)}\rangle-\frac{H}{2}|\phi_*^{(z)}|^2\right)d\Sigma_\epsilon\\
& = & \lim_{\epsilon\to 0}\int_{\Sigma_\epsilon}\left(\langle \widetilde D^{\Sigma_\epsilon}\phi^{(z)}_*,\phi^{(z)}_*\rangle-\frac{H_0}{2}|\phi^{(z)}_*|^2\right)d\Sigma_\epsilon\\
& & \quad 
+\lim_{\epsilon\to 0}\frac{1}{2}\int_{\Sigma_\epsilon}\left(H_0-H\right)|\phi^{(z)}_*|^2d\Sigma_\epsilon, 
\end{eqnarray*}
where in the second step we used that standard cancellations imply that in the limit the boundary term only captures the contribution coming from the terms which are quadratic in $\phi^{(z)}_*$.

Recalling that
the label $0$ refers to invariants of $g_0$ and its spin bundle we note that  $\widetilde\Wc^0=-(\widetilde\nabla^0_{\nu^0}+
\gamma^0(\nu^0)\widetilde D_0$), where $\nu^0$ is the inward point unit normal to $\Sigma_\epsilon\subset\mathbb H^3$,
is self-adjoint (in the $L^2$ sense) when acting on spinors restricted to $\Sigma_\epsilon$. Thus, since
$\phi^{(z)}_*=\Bc\phi^{(z)}=\phi^{(z)}+\Kc\phi^{(z)}$, $|\Kc|=O(\epsilon^3)$, we have 
\begin{eqnarray*}
\int_{\Sigma_\epsilon}\left(\langle \widetilde D^{\Sigma_\epsilon}\phi^{(z)}_*,\phi^{(z)}_*\rangle-\frac{H_0}{2}|\phi^{(z)}_*|^2\right)d\Sigma_\epsilon
& = & {\rm Re}\int_{\Sigma_\epsilon}\left\langle
\widetilde\Wc^0\phi^{(z)}_*,\phi^{(z)}_*
\right\rangle d\Sigma_\epsilon\\
& = &  {\rm Re}\int_{\Sigma_\epsilon}\left\langle
\widetilde\Wc^0\Bc\phi^{(z)},\Bc\phi^{(z)}
\right\rangle d\Sigma_\epsilon\\
& = & {\rm Re}\int_{\Sigma_\epsilon}\left\langle
\Bc\phi^{(z)},\widetilde\Wc^0\phi^{(z)}
\right\rangle d\Sigma_\epsilon\\
&  & \quad + {\rm Re}\int_{\Sigma_\epsilon}\left\langle
\widetilde\Wc^0\Bc\phi^{(z)},\Kc\phi^{(z)}
\right\rangle d\Sigma_\epsilon.
\end{eqnarray*}
But $\widetilde\Wc^0\phi^{(z)}=0$ while by (\ref{gauge})  we have 
$|\widetilde \Wc^0\Bc|=O(\epsilon^3)$, so 
we find that the above integral is $O(\epsilon^3)$,
which gives
\[
\int_M\left(|\widetilde \nabla\psi^{(\eta)}|^2+
\frac{R_g+6}{4}|\psi^{(\eta)}|^2\right)dM
    =
    \lim_{\epsilon\to 0}\frac{1}{2}\int_{\Sigma_\epsilon}\left(H_0-H\right)|\phi^{(z)}|^2 d\Sigma_\epsilon,
\]
where we used that $|\phi^{(z)}_*|=|\phi^{(z)}|$.
If we now appeal to 
(\ref{gives}) this proves (\ref{newproof2}) and completes the proof of Theorem \ref{mmain3}.

It is obvious that (\ref{newproof2}) implies the inequality (\ref{limit}) in Theorem \ref{mmain1} if $R_g\geq -6$. As for the rigidity statement, if the equality holds for some $\eta$
then from (\ref{newproof2}) we get
\[
\int_M|\widetilde\nabla \psi^{(\eta)}|^2dM\leq 0,
\]
that is, $\psi^{(\eta)}$ is an imaginary Killing spinor. A well-known result by Baum \cite{BFGK} \cite{AD} then implies that $(M,g)$ is isometric to $(\mathbb H^3,g_0)$, as desired.

\begin{remark}
\label{lamb}
{\rm 
If we assume that $(M,g)$ carries a compact inner boundary $\Gamma$ then instead of (\ref{newproof2}) we now have 
\begin{eqnarray*}
\int_M\left(|\widetilde \nabla^+\psi^{(\eta)}|^2+\frac{R_g+6}{4}|\psi^{(\eta)}|^2\right)dM
& = & -4\pi\langle\langle \lim_{\epsilon\to 0}\mathfrak m_{BY}(\Sigma_\epsilon,X_\epsilon),\eta\rangle\rangle \\
  & & \quad +\int_{\Gamma}\left\langle \left(\widetilde D^\Gamma_+-\frac{\Hc}{2}\right)\psi^{(\eta)},\psi^{(\eta)}\right\rangle d\Gamma,
\end{eqnarray*}
where 
$\Hc$ is the mean curvature of $\Gamma$ with respect to $\nu_\Gamma$, the unit normal pointing toward infinity, and 
$\widetilde D_+^{\Gamma}=D^\Gamma+i\gamma(\nu_\Gamma)$.
By \cite{BC} we may choose $\psi^{(\eta)}$ so that it satisfies the MIT bag boundary condition $i\gamma(\nu_\Gamma)\psi^{(\eta)}=-\psi^{(\eta)}$ along $\Gamma$. Hence, by (\ref{cancell}),
\begin{eqnarray*}
\int_M\left(|\widetilde \nabla^+\psi^{(\eta)}|^2+\frac{R_g+6}{4}|\psi^{(\eta)}|^2\right)dM
& = & -4\pi\langle\langle \lim_{\epsilon\to 0}\mathfrak m_{BY}(\Sigma_\epsilon,X_\epsilon),\eta\rangle\rangle \\
  & & \quad -\int_{\Gamma} \left(\frac{\Hc+2}{2}\right)|\psi^{(\eta)}|^2d\Gamma.
\end{eqnarray*}
Thus, if $R_g\geq -6$ and $\Hc\geq -2$ we see that 
\[
\langle\langle \lim_{\epsilon\to 0}\mathfrak m_{BY}(\Sigma_\epsilon,X_\epsilon),\eta\rangle\rangle\leq 0, \quad \eta\in\Cc^+,
\]
just as in Theorem \ref{mmain1}. Moreover, if the equality holds for some $\eta$ then $\psi^{(\eta)}$ is an imaginary Killing spinor and hence $(M,g)$ is Einstein, ${\rm Ric}_g=-2g$. Moreover, there holds $\Hc\equiv -2$ so if we use that by (\ref{asymp1}) the mean curvature of the coordinate spheres $\Sigma_\epsilon$, computed with respect to the unit normal pointing toward infinity, converges to $-2$ as $\epsilon\to 0$ and argue as in the proof of \cite[Theorem 4.7]{CH} we easily reach a contradiction. This shows that under the conditions of Theorem \ref{mmain1} and in the presence of an inner boundary $\Gamma$ as above we always have that $\lim_{\epsilon\to 0}\mathfrak m_{BY}(\Sigma_\epsilon,X_\epsilon)$ is time-like and future directed. In other words, this type of {\em trapped} inner boundary which, under suitable global conditions, foretells the existence of a black hole region in the Cauchy development of $(M,g)$, effectively contributes to force this vector to be  time-like. This clearly suggests that a Penrose-type inequality might hold in this setting. Notice  also that a similar statement holds in the context of Theorem \ref{main1}. We thank L. Ambrozio for enlightening conversations regarding this interesting sharpening of our results.
}
\end{remark}

\section{The proofs of Theorems \ref{main1}, \ref{main2} and \ref{main3}}\label{proofrest}

In this section we present the proofs of Theorems \ref{main1}, \ref{main2} and {\ref{main3}}. As remarked in the Introduction, these results  only depend on the theory of boundary value problems for Dirac operators in compact domains as described in Section \ref{dirac} and Appendix \ref{appendixA}. 

\begin{proposition}\label{mainhyp}
Let $(\Omega,g)$ be a compact $3$-manifold with boundary a sphere $\Sigma$ satisfying $K>-1$ and $H>-2$.
Then there holds
\begin{equation}\label{mainhypineq}
\int_{\Sigma} \frac{H_0^2-H^2}{H+2}|\hat\phi|^2 d\Sigma+4\int_{\Sigma}\frac{\Delta_{\Sigma}|\hat\phi|^2}{H+2}
d\Sigma \geq 0,
\end{equation}
for any imaginary Killing spinor $\hat\phi\in\Gamma(\mathbb S\mathbb H^3)$, where $H_0$ is the mean curvature of the embedding $\Sigma\subset \mathbb H^3$. Moreover, equality holds if and only if the shape operators of the embeddings $\Sigma\subset\Omega$ and $\Sigma\subset\mathbb H^3$ coincide.
\end{proposition}

\begin{proof}
Since $\Sigma$ is topologically a sphere, it carries a {\em unique} spin structure. This allows us to take $\phi=\hat\phi$  in (\ref{estimateq}).
It follows from (\ref{form1}) and (\ref{form2}) that
\[
D^\Sigma\hat\phi=\frac{H_0}{2}\hat\phi-i\gamma^0(\nu^0)\hat\phi,
\]
which gives
\begin{eqnarray*}
|D^\Sigma \hat\phi|^2 & = & \left(\frac{H_0^2}{4}+1\right)|\hat\phi|^2-H_0\langle i\gamma^0(\nu^0)\hat\phi,\hat\phi\rangle \\
 & = & \left(\frac{H_0^2}{4}+1\right)|\hat\phi|^2+2H_0\langle \nabla^0_{\nu^0}\hat\phi,\hat\phi\rangle \\
  & = & \left(\frac{H_0^2}{4}+1\right)|\hat\phi|^2+H_0{\nu^0}(|\hat\phi|^2).
\end{eqnarray*}
Now, it follows easily from (\ref{baum1}) that the restriction of $|\hat\phi|^2$ to $\Sigma$ satisfies the
Minkowski-type identity
\[
\Delta_\Sigma|\hat\phi|^2=2|\hat\phi|^2+H_0\nu^0(|\hat\phi|^2),
\]
so we get
\begin{equation}\label{asin}
|D^\Sigma \hat\phi|^2  =  \left(\frac{H_0^2}{4}-1\right)|\hat\phi|^2 +\Delta_\Sigma|\hat\phi|^2,
\end{equation}
which proves (\ref{mainhypineq}). The rigidity statement follows essentially by the same argument as in the proof of \cite[Theorem 2]{HM}, so we omit it here.
\end{proof}

Under the conditions of Theorem \ref{main1}, the second integral in (\ref{mainhypineq}) vanishes so we obtain
\begin{equation}\label{mainhypineq2}
\int_{\Sigma} \frac{H_0^2-H^2}{H+2}|\hat\phi|^2 d\Sigma
\geq 0,
\end{equation}
with the equality occurring if and only if $\Omega\subset\mathbb H^3$ is a domain whose boundary has constant mean \ curvature, i.e. $\Omega$ is a round ball. 
Since, as remarked in the previous section, 
$\overline \eta$  maps $\mathbb C^2$ onto $\Cc^+$,  we can use (\ref{mainhypineq2}) with $\hat\phi=\phi^{(z)}$ and (\ref{gives})
to conclude the proof of Theorem \ref{main1}.

We now show how Theorems \ref{main2}
and \ref{main3} follow
from two auxiliary results  describing the asymptotic behavior of the terms in the left-hand side of (\ref{mainhypineq})
along a coordinate exhaustion of an asymptotically hyperbolic manifold.

\begin{proposition}\label{funclim0}
If $(\Omega_\epsilon,\Sigma_\epsilon)$ is a coordinate exhaustion of an asymptotically hyperbolic manifold then
\begin{equation}\label{funclim20}
\lim_{\epsilon\to 0}\int_{\Sigma_\epsilon}\frac{H_0^2-H^2}{H+2}|\hat\phi|^2d\Sigma_\epsilon=\lim_{\epsilon\to 0}\int_{\Sigma_\epsilon}(H_0-H)|\hat\phi|^2d\Sigma_\epsilon,
\end{equation}
for any imaginary Killing spinor $\hat\phi\in\Gamma(\mathbb S\mathbb H^3)$.
\end{proposition}

\begin{proof}
It is shown in \cite{KT} that 
\begin{equation}\label{asymp1}
H=2\cosh \epsilon -\frac{\epsilon^3}{2}{\rm tr}_{h_0}h+O(\epsilon^4)=2+\epsilon^2 -\frac{\epsilon^3}{2}{\rm tr}_{h_0}h+O(\epsilon^4)
\end{equation}
and
\begin{equation}\label{asymp2}
H_0=2\cosh \epsilon +O(\epsilon^5)=2+{\epsilon^2}+o(\epsilon^4).
\end{equation}
This
readily gives
\[
\frac{H_0+2}{H+2}=1+O(\epsilon^3),
\]
so that
\begin{equation}\label{asymm}
\int_{\Sigma_\epsilon}\frac{H_0^2-H^2}{H+2}|\hat\phi|^2d\Sigma_\epsilon  =
     \int_{\Sigma_\epsilon}(H_0-H)|\hat\phi|^2d\Sigma_\epsilon +\int_{\Sigma_\epsilon}O(\epsilon^5)|\hat\phi|^2d\Sigma_\epsilon.
\end{equation}
From (\ref{areagrowth}) and (\ref{partic})
we have
\[
\int_{\Sigma_\epsilon}O(\epsilon^5)|\hat\phi|^2d\Sigma_\epsilon=O(\epsilon^2),
\]
so that (\ref{funclim20}) follows from (\ref{asymm}).
\end{proof}

\begin{proposition}\label{funclim}
If $(\Omega_\epsilon,\Sigma_\epsilon)$ is a coordinate exhaustion of an asymptotically hyperbolic manifold  then
\begin{equation}\label{funclim21l}
\lim_{\epsilon\to 0}\int_{\Sigma_\epsilon}\frac{\Delta_{\Sigma_{\epsilon}}|\hat\phi|^2}{H+2}d\Sigma_\epsilon=0,
\end{equation}
for any imaginary Killing spinor $\hat\phi\in\Gamma(\mathbb S\mathbb H^3)$.
\end{proposition}

\begin{proof}
It follows from  (\ref{asymp2})
that
\[
\frac{1}{H+2}=\frac{1}{4}-\frac{\epsilon^2}{16}+\frac{\epsilon^3}{32}{\rm tr}_{h_0}h+O(\epsilon^4),
\]
which
gives
\begin{equation}\label{prov}
\int_{\Sigma_\epsilon}\frac{\Delta_{\Sigma_\epsilon}|\hat\phi|^2}{H+2}d\Sigma_\epsilon  =  \frac{\epsilon^3}{32}
    \int_{\Sigma_\epsilon}{\rm tr}_{h_0}h\Delta_{\Sigma_\epsilon}|\hat\phi|^2d\Sigma_\epsilon + \int_{\Sigma_\epsilon}O(\epsilon^4)\Delta_{\Sigma_\epsilon}|\hat\phi|^2d\Sigma_\epsilon.
\end{equation}
Using (\ref{meteps}) and
(\ref{partic})
we obtain
$\Delta_{\Sigma_\epsilon}|\hat\phi|^2=\epsilon^2\Delta_{h_0}|\hat\phi|^2+O(\epsilon^2)$.
Combining this with  (\ref{areagrowth}) and (\ref{partic}) we get
\begin{eqnarray*}
\int_{\Sigma_\epsilon}{\rm tr}_{h_0}h\Delta_{\Sigma_\epsilon}|\hat\phi|^2d\Sigma_\epsilon
    & = & \int_{\mathbb S^2}{\rm tr}_{h_0}h\Delta_{h_0}|\hat\phi|^2d\mu_{h_0} + O(1)\\
    & = & -\int_{\mathbb S^2}\langle \nabla_{h_0}{\rm tr}_{h_0}h,\nabla_{h_0}|\hat\phi|^2\rangle d\mu_{h_0} + O(1).
\end{eqnarray*}
On the other hand, by (\ref{baum2}) we also have
\begin{eqnarray*}
|\nabla_{h_0}|\hat\phi|^2| & = & \epsilon^{-1}|\nabla_{\Sigma_\epsilon}|\hat\phi|^2|+O(\epsilon^{-2})\\
  & \leq &  \epsilon^{-1} |\nabla^0|\hat\phi|^2|+ O(\epsilon^{-2}) \\
  & = & O(\epsilon^{-2}),
\end{eqnarray*}
so that, by Cauchy-Schwarz, the first term in the right-hand side of (\ref{prov}) is $O(\epsilon)$. A similar argument shows that the second one is $O(\epsilon^2)$,
so the assertion follows.
\end{proof}

\appendix\section{The proof of Proposition \ref{bvpmr}}
\label{appendixA}

In this appendix we reproduce the argument in \cite{R} leading to the proof of Proposition \ref{bvpmr}. We start by recalling the integration by parts formula for the Dirac operator $D$ of a compact, spin $3$-manifold $(\Omega,g)$ with boundary $\Sigma$ oriented by its inward unit vector $\nu$:
\begin{equation}\label{green}
\int_\Omega\langle D\varphi,\psi\rangle dM-\int_\Omega\langle \varphi ,D\psi\rangle dM=-\int_\Sigma\langle \gamma(\nu)\phi,\psi\rangle d\Sigma, \quad \varphi,\psi\in\Gamma(\mathbb S\Omega).
\end{equation}

\begin{lemma}\label{bvp1}
Under the conditions above, the adjoint operator (in the $L^2$ sense) of the operator $\widetilde D^{\pm}$ with domain
\[
{\rm dom}\,\widetilde D^\pm=\left\{\psi\in\Gamma(\mathbb S\Omega); \psi_\mp=0\,{\rm on}\,\Sigma\right\},
\]
is
the operator $\widetilde D^{\mp}$ with domain
\[
{\rm dom}\,\widetilde D^\mp=\left\{\psi\in\Gamma(\mathbb S\Omega); \psi_\pm=0\,{\rm on}\,\Sigma\right\}.
\]
\end{lemma}

\begin{proof}
Take $\varphi\in {\rm dom}\,\widetilde D^\pm$ and $\psi\in\Gamma(\mathbb S\Omega)$. From (\ref{green}) we have
\[
\int_\Omega\langle \widetilde D^\pm\varphi,\psi\rangle dM-\int_\Omega\langle \varphi ,\widetilde D^\mp\psi\rangle dM=-\int_\Sigma\langle \gamma(\nu)\varphi,\psi\rangle d\Sigma=\mp\int_\Sigma\langle i\varphi,\psi\rangle d\Sigma,
\]
so if $\varphi$ varies over the space of spinors vanishing on $\Sigma$, we see that
\[
(\widetilde D^{\pm})^*\psi=\widetilde D^\mp\psi, \quad \psi\in {\rm dom}\,\widetilde D^\pm.
\]
As a consequence, $\int_\Sigma\langle i\varphi,\psi\rangle d\Sigma=0$ for any $\varphi\in\Gamma(\mathbb S\Omega)$ with $\varphi_\mp=0$, which gives $\psi_\pm=0$, as desired.
\end{proof}

\begin{lemma}\label{bvp2}
If $R_g\geq -6$ and $H\geq -2$ then at least one of the following assertions holds true:
\begin{enumerate}
 \item The homogeneous boundary value problem
 \[
\left\{
\begin{array}{rclc}
\widetilde D^\pm\psi & = &  0 & {\rm on}\,\Omega \\
\psi_\pm & = &  0 & {\rm on}\,\Sigma
\end{array}
\right.
\]
has no nontrivial solution;
  \item There exists a nontrivial imaginary Killing spinor on $\Omega$ and $H\equiv -2$ on $\Sigma$.
\end{enumerate}
\end{lemma}

\begin{proof}
If $\psi$ is a nontrivial solution of the boundary value problem then from (\ref{reillyhyp2})
we obtain
\begin{equation}\label{port}
\int_\Sigma\left(\langle \widetilde D^\Sigma_\pm\psi,\psi\rangle-\frac{H}{2}|\psi|^2\right)d\Sigma\geq 0.
\end{equation}
But the boundary condition clearly implies
\[
\int_\Sigma \langle \widetilde D^{\Sigma}_\pm\psi,\psi\rangle d\Sigma=-\int_\Sigma |\psi_\mp|^2d\Sigma,
\]
that is, equality holds in (\ref{port}). In this case, $R_g\equiv -6$ and $\psi$ is a twistor spinor on $\Omega$. Since $\psi$ is also an eigenvector of $D$, we conclude that $\psi$ is in fact an imaginary Killing spinor. It is well-known  that $\psi$ has no zeros \cite{BFGK}  and this implies that $H\equiv -2$ along $\Sigma$.
\end{proof}

\begin{lemma}\label{bvp3}
If $R_g\geq -6$ and $H\geq -2$ then at least one of the following assertions holds true:
\begin{enumerate}
 \item The nonhomogeneous boundary value problem
 \[
\left\{
\begin{array}{rclc}
\widetilde D^\pm\psi & = &  0 & {\rm on}\,\Omega \\
\psi_\pm & = &  \phi_\pm & {\rm on}\,\Sigma
\end{array}
\right.
\]
has a unique solution for each $\phi\in\Gamma(\mathbb S\Sigma)$;
  \item There exists a nontrivial imaginary Killing spinor on $\Omega$ and $H\equiv -2$ on $\Sigma$.
\end{enumerate}
\end{lemma}

\begin{proof}
Ellipticity of  $\widetilde D^\pm$ implies that its realization under the given boundary value conditions is Fredholm \cite[Chapter 19]{BoW}. If the first item in the previous lemma holds then
\[
\ker (\widetilde D^\pm,P_\pm)=\{0\},
\]
so that
\[
{\rm coker}\, (\widetilde D^\pm,P_\pm)^*=\{0\}.
\]
By Lemma \ref{bvp1}, this means that
\[
{\rm coker}\, (\widetilde D^\mp,P_\mp)=\{0\},
\]
that is, the first item in the lemma holds.
\end{proof}

Proposition \ref{bvpmr} is an immediate consequence of Lemma \ref{bvp3}.

%
%




\end{document}